\pdfoutput=1
\pdfoutput=1
\pdfoutput=1
\pdfoutput=1
\pdfoutput=1
\documentclass[a4paper,10pt]{article}
\usepackage{calc}
\usepackage[all]{xy}
\usepackage[centertags]{amsmath}
\usepackage{latexsym}
\usepackage{amsfonts}
\usepackage{graphicx}
\usepackage{tikz}
\usepackage{cases}
\usepackage{amssymb}
\usepackage{amsthm}
\usepackage{color}
\usepackage{fancyhdr}
\usepackage[dvips]{epsfig}
\usepackage{newlfont}
\usepackage[latin1]{inputenc}
\usepackage[latin1]{inputenc}
\usepackage[english,french]{babel}
\usepackage{graphicx}
\usepackage{t1enc}
\usepackage[english,french]{babel}
\usepackage{fancybox}
\usepackage{graphicx}
\usepackage{t1enc}
\usepackage[french2]{minitoc}
\usepackage{mathrsfs}
\usepackage{amsfonts}
\usepackage{wasysym}
\usepackage{hyperref}
\allowdisplaybreaks
\usepackage{float}
\usepackage{authblk}
	\numberwithin{equation}{section} 
	\theoremstyle{plain} 
	\newtheorem{theorem}{Theorem}[section]
	\newtheorem{lemma}[theorem]{Lemma}
	\newtheorem{remark}[theorem]{Remark}
	\newtheorem{proposition}[theorem]{Proposition}
	\newtheorem{fact}[theorem]{Fact}
	\newtheorem{definition}[theorem]{Definition}

	 \makeatletter
	\newcommand{\thechapterwords}
	{ \ifcase \thechapter\or 1\or 2\or 3\or 4\or 5\or
		6\or 7\or 8\or 9\or 10\or 11\fi}
	\def\thickhrulefill{\leavevmode \leaders \hrule height 2ex \hfill \kern \z@}
	\def\@makechapterhead#1{%
		\vspace*{15\p@}%
		{\parindent \z@ \centering \reset@font
			\thickhrulefill\quad
			\scshape  {\chapnumfont \@chapapp{}}{\chapnumfont \thechapterwords}
			\quad \thickhrulefill
			\par\nobreak
			\vspace*{15\p@}%
			\interlinepenalty\@M
			\hrule
			\vspace*{15\p@}%
			\huge {\bfseries  #1}\par\nobreak
			\par
			\vspace*{15\p@}%
			\hrule
			\vskip 15\p@
	}}
	\def\@makeschapterhead#1{%
		\vspace*{15\p@}%
		{\parindent \z@ \centering \reset@font
			\thickhrulefill
			\par\nobreak
			\vspace*{15\p@}%
			\interlinepenalty\@M
			\hrule
			\vspace*{15\p@}%
			\Huge \bfseries #1\par\nobreak
			\par
			\vspace*{15\p@}%
			\hrule
			\vskip 30\p@
	}}
	
	\DeclareFixedFont{\chapnumfont}{T1}{phv}{b}{n}{20pt}
	\DeclareFixedFont{\chapchapfont}{T1}{phv}{b}{n}{16pt}
	\DeclareFixedFont{\chaptitfont}{T1}{phv}{b}{n}{24.88pt}
	\def\@makechapterhead#1{%
		\vspace*{15\p@}%
		{\parindent \z@ \centering \reset@font
			\thickhrulefill\quad
			\scshape {\chaptitfont\color[rgb]{0.00,0.50,1.00}\@chapapp{}}
			{\chapnumfont \thechapterwords}
			\quad \thickhrulefill
			\par\nobreak
			\vspace*{15\p@}%
			\interlinepenalty\@M
			\hrule
			\vspace*{15\p@}%
			{\Large\bfseries #1}\par\nobreak
			\par
			\vspace*{15\p@}%
			\hrule
			\vskip 30\p@
	}}%
	
		\title{$\theta$-almost twisted Poisson cohomology}
\author[1]{Nasser Saipele Nansidi}
\author[2,3]{Bertuel Tangue Ndawa}
\author[1]{Joseph Dongho}
\affil[1]{Faculty of Science, University of Maroua, Cameroon.}
\affil[2]{University Institute of Technology, University of Ngaoundere, Cameroon.}
\affil[3]{Institut des Hautes Études Scientifiques, Université Paris-Saclay, France.}
	\date{}
	
\begin{document}
	\maketitle
	\selectlanguage{english}

	%\begin{document}
	
	%\titlerunning{ $\theta$-almost twisted Poisson cohomology} % for running heads
	%\authorrunning{Nasser Saipele Nansidi at  al.} % for running heads
	%%\authorrunning{First-Author, Second-Author} % for running heads
	%
	%\title{On the cohomology of $\theta$-almost twisted Poisson \\
		%structure}
	%% Splitting into lines is performed by the command \\
	%% The title is written in accordance with the rules of capitalization.
	%
	%\author{\firstname{Nasser}~\surname{  Saip\'el\'e Nansidi}}
	%\email[E-mail: ]{nankamla@gmail.com}
	%\affiliation{Department of Mathematics and Computer Science, University of Maroua, Cameroon}
	%
	%\author{\firstname{Bertuel}~\surname{Tangue Ndawa}}
	%\email[E-mail: ]{bertuel@yahoo.fr}
	%\affiliation{Department of Computer Engineering, University of Ngaoundere, Cameroon}
	%%\noaffiliation % If the author does not specify a place of work.
	%
	%\author{\firstname{Joseph}~\surname{Dongho}}
	%\email[E-mail: ]{joseph.dongho75@gmail.com}
	%\affiliation{Department of Mathematics and Computer Science, University of Maroua, Cameroon}
	%\firstcollaboration{(Submitted by A.~A.~Editor-name)} % Add if you know submitter.
	%\lastcollaboration{}
	
	%\received{} % The date of receipt to the editor, i.e. December 06, 2017

	%\begin{abstract}
	\paragraph{Abstract:}
	We introduce the notion of a $\theta$-almost twisted Poisson structure on manifolds, which involves incorporating a closed $1$-form $\theta$ into twisted Poisson structures under specific conditions. We provide a characterization of this structure on low-dimensional manifolds and construct the Lie-Rinehart algebra on the module of $1$-forms on manifolds equipped with this structure. This construction leads to a cochain complex and its associated cohomology, which we refer to as $\theta$-almost twisted Poisson cohomology. An example illustrating this cohomology is also presented on $\mathbb{R}^5$.\vspace{0.5cm}
	%\end{abstract}

	
	\section{Introduction}
	
	Poisson structures, first introduced by Sim\'{e}on Denis Poisson in the early 19th century, form a cornerstone of symplectic geometry and Hamiltonian mechanics. They are defined by a bivector field $\pi$ on a differentiable manifold $M$ satisfying $[\pi,\pi]=0$, where $[\ ,\ ]$ denotes the
	Shouten-Nijenhuis bracket, see \cite{R4}. These structures encode symmetries and conservation laws in physical systems, and their quantization has profoundly influenced deformation theory and noncommutative geometry, see \cite{R10,R7}.

	During the  21st century, a significant generalization emerged under the name of twisted Poisson structures. Introduced by P. \~Severa et al \  \cite{R11},
	these structures relax the Jacobi identity  $[\pi,\pi]=0$  by twisting it with a closed 3-form $\varphi$ verifying
	$\dfrac{1}{2}[\pi,\pi]=\pi^{\#}  (\varphi)$  where  $\pi^{\#}$  denotes the anchor map associated with $\pi$, see Definition  \ref{expPi}.
	This geometric torsion is inspired by the works of C. Klim\v{c}\'ik et al. \cite{R14}, J.S. Park \cite{R12}, and
	L. Cornalba et al. \cite{R13} on deformation quantization and string theory, in which such a $3$-form $\varphi$ plays an important role,
	particularly in compactifications with non-trivial fluxes or non-associative geometries.
	
	Note that it can arise in string theories, more particularly in heterotic string theories, that the $3$-form $\varphi$ loses its closure. This is the case where it acquires a $\alpha'$ correction \cite{R17}. Such cases can also arise in the reduction of Courant algebroids \cite{R18}.
	In 2020, A. Chatzistavrakidis  studied aspects of two-dimensional nonlinear sigma models with a Wess-Zumino term corresponding to
	a nonclosed $3$-form $\varphi$ which may arise upon dimensional
	reduction in the target space, see \cite{R15}. Such a $3$-form $\varphi$ is of the form $d\varphi = -G \wedge F$, where $F$ is a $2$-form and $G$ is a closed $2$-form.
	
	Motivated by the works of P. \v{S}evera et al. \cite{R11} and A. Chatzistavrakidis \cite{R15}, we extend the classical framework of twisted Poisson structures on a manifold $M$ by introducing a closed $1$-form $\theta$ coupled to the $3$-form $\varphi$ via the conditions $d\varphi = \varphi \wedge \theta$ and $\pi^{\#}(\varphi) = 0$, and the quadruplet $(M,\pi,\varphi,\theta)$ is called
	a $\theta$-almost twisted Poisson manifold, see Definition~\ref{def1}. This modification preserves the fundamental equation $\frac{1}{2}[\pi,\pi] = \pi^{\#}(\varphi)$ while allowing
	$\varphi$ to no longer necessarily be closed.
	%, instead governed by $\theta$. 
	%, see Definition \ref{def1}. 
	This approach aligns with work on twisted Jacobi structures \cite{R6}
	and locally conformally cosymplectic Hamiltonian dynamics and Hamilton-Jacobi theory \cite{R20}. On the physical front, such structures are relevant for modeling generalized geometric fluxes, particularly in string theories and quantum gravity. 
	%For instance, in flux compactification, the $3$-form $\varphi$ and $1$-form $\theta$ could correspond to combined contributions of NS-NS and R-R fluxes or high-energy corrections, see \cite{R16,R15}.
	Despite the equivalence in dimensions less than $5$, the $\theta$-almost twisted Poisson structures are a generalization of twisted Poisson
	structures. 
	
	For a $\theta$-almost twisted Poisson manifold $(M, \pi, \varphi, \theta)$, the bracket $\{f, g\} = \pi(df, dg)$, $f, g \in C^\infty(M)$, does not satisfy the Jacobi identity. Its associated Jacobi operator acquires an extra term involving $\varphi$, so conventional methods for cohomology cannot be applied. Moreover, when $\varphi$ is not closed, the pair $(C^\infty(M, \mathbb{R}), \{\cdot, \cdot\})$ is neither a Lie algebra nor a $\varphi$-twisted Lie algebra (see \cite{R2}). There is a large class of such manifolds for which $0\neq d\varphi = \theta \wedge \varphi$ ($\varphi$ is not necessarily closed). For this class, the mechanism for constructing the Lie-Rinehart algebra (Lie algebroid) in \cite{R2} is not applicable. For this reason, we introduce a new bracket in order to describe a cohomology for general $\theta$-almost twisted Poisson manifolds.
	%  see $\S$\ref{TATPs1}.
	
	%This result significantly influences the description of the cohomology of the $\theta$-almost twisted Poisson manifold, .  
	%
	%no longer works
	% due to the presence of the
	%closed $1$-form $\theta$, making it necessary to modify this machinery.
	Before we can explain our results more precisely, we need to fix some notations, present some basic definitions, and recall some known results that will be useful in the sequel.
	%In line to this, the work is organized as follow: We first present some definitions, and formulate certain known results necessary
	%for the subsequent sections.
	%Section 3 is devote to  introduce the notion of   $\theta$-almost twisted Poisson  manifold,  to give some  examples of such
	%manifolds, and to  show that these kind of manifolds are equivalent to twisted Poisson manifolds in dimensions less than 5.
	%We introduce in section 4  the notion of $\theta$-almost twisted Poisson cohomology.  This is possible by  constructing  the
	%Lie-Rinehart algebra on 1-forms module on $\theta$-almost twisted Poisson  manifold.
	%We end this article by an example of   $\theta$-almost twisted Poisson cohomology calculation.
	\section{Tools}
	Throughout this paper, $k \geq 1$, all objects are smooth, and $M$ denotes a manifold. For every $f\in C^\infty(M)$ a real smooth function on $M$, if $f$ is nowhere vanishing, $f^{-1}$ stands for $1/f$.  For $p \geq 0$, $\Omega^p(M)$ and $\mathfrak{X}^p(M)$ denote the set of $p$-forms and $p$-vector fields on $M$, respectively. In particular, $\Omega^0(M) = \mathfrak{X}^0(M) = C^\infty(M)$. We denote $\Omega(M) = \bigoplus_{p \geq 0} \Omega^p(M)$ and $\mathfrak{X}(M) = \bigoplus_{p \geq 0} \mathfrak{X}^p(M)$, where $\bigoplus$ denotes the direct sum.
	
	Let $X$ be a vector field on $M$. The interior derivative $i_X$ by $X$ is the unique graded $C^\infty(M)$-linear endomorphism of $\Omega(M)$, of degree $-1$, such that for every $\varphi \in \Omega^k(M)$, $i_X\varphi$ is the element of $\Omega^{k-1}(M)$ defined by
	\begin{align*}
		i_X\varphi(X_1,\dots,X_{k-1})&=\varphi(X,X_1,\dots,X_{k-1}), \; X_1,\dots,X_{k-1}\in \Omega^1(M).
	\end{align*}
	A bivector  $\pi\in \mathfrak{X}^{2}(M)$ induces the morphism 
	{$\pi^{\#}: \Omega^1(M)\longrightarrow \mathfrak{X}^1(M),\ \alpha \longmapsto  \pi^{\#}(\alpha) $}  where
	\begin{align}\label{expPi}
		% \nonumber % Remove numbering (before each equation)
		\pi^{\#}(\alpha)(\beta)&=\pi(\alpha,\beta),\ \ \beta\in \Omega(M).
	\end{align}
	The map $\pi^{\#}$, which can be extended from $\Omega^k(M)$ to $\mathfrak{X}^k(M)$, can be defined by
	\begin{align}
		\pi^{\#}(\varphi)(\alpha_1,\dots,\alpha_k) &= (-1)^{k}\varphi\bigg(\pi^{\#}(\alpha_1),\dots,\pi^{\#}(\alpha_k)\bigg),
	\end{align}
	for every $\varphi\in\Omega^k(M)$ and $\alpha_1,\dots,\alpha_k\in \Omega^1(M)$.
	In particular, $\pi^{\#}(f) = f$ for $f \in C^\infty(M)$.
	
	For every differential form $\varphi$ on $M$, $\varphi^{\#}$ stands for $\pi^{\#}(\varphi)$.
	
	\begin{theorem}\cite[Theorem 2.8]{R9}\label{ui}
		There exists a  unique  $\mathbb{R}$-bilinear operation  $[\ ,\ ]: \mathfrak{X}^{p+1}(M)\times \mathfrak{X}^{q+1}(M)\longrightarrow \mathfrak{X}^{p+q+1}(M),  (p,q\geq -1)$ \\
		which satisfies  the following properties:
		\begin{itemize}
			\item [(i)] When $p=q=0$, it is the usual Lie bracket of vector fields;
			\item [(ii)] When $p=0$ and $q=-1$, it is the Lie derivative $[X,f]=\mathcal{L}_Xf=X(f)$;\\
			\item[\mbox{}]For $P\in\mathfrak{X}^{p+1}(M)$, $Q\in \mathfrak{X}^{q+1}(M)$,  and $R\in \mathfrak{X}^{r+1}(M)$,
			\item [(iii)] Graded skew-symmetry: \ \  {$[P,Q]=-(-1)^{pq}[Q,P]$;}
			
			\item[(iv)] Graded Leibniz identity:\ \   $[P,Q\wedge R]=[P,R]\wedge R+(-1)^{p(q+1)}Q\wedge[P,R];$
			\item[(v)]  Graded Jacobi identity: the Jacobi operator (Jacobiator) $J_{[,]}(P,Q,R)$ of $P,Q,R$ is identically null, where
			$$
			J_{[,]}(P,Q,R) = (-1)^{rp}[P,[Q,R]] + (-1)^{pq}[Q,[R,P]] + (-1)^{qr}[R,[P,Q]].
			$$
		\end{itemize}
	\end{theorem}
	
	\begin{definition} The bracket $[\ ,\ ]: \mathfrak{X}^{p+1}(M)\times \mathfrak{X}^{q+1}(M)\longrightarrow \mathfrak{X}^{p+q+1}(M),  (p,q\geq -1)$ defined in Theorem \ref{ui}  is called  Schouten-Nijenhuis bracket.
	\end{definition}
	
	\begin{definition}\label{expi2}\cite[page 2-3]{R2}
		A twisted Poisson manifold is a manifold $M$ endowed with a bivector field $\pi$ and a closed 3-form $\varphi$ on $M$ such that
		\begin{align*}
			% \nonumber % Remove numbering (before each equation)
			\dfrac{1}{2}[\pi,\pi] &=\pi^{\#}  (\varphi).
		\end{align*}
	\end{definition}
	\begin{proposition}\cite[Proposition 2.1]{R20}\label{prop1}
		Let $(M,\pi,\varphi)$ be a twisted Poisson manifold. For every functions $f,g \in C^\infty(M)$, one has
		\begin{align*}
			\left[(df)^{\#},(dg)^{\#}\right]-\big(d\{f,g\}\big)^{\#} &= (i_{(dg)^{\#}}i_{(df)^{\#}}\varphi)^{\#}.
		\end{align*}
		%where $[(df)^{\#},(dg)^{\#}]$ denotes the Lie bracket of vector fields $(df)^{\#}$ and $(dg)^{\#}$.
	\end{proposition}
	
	Let $R$ be a commutative ring. Let $A$ be an algebra over $R$.
	We recall that a derivation of $A$ is a morphism $\delta: A\longmapsto A$ of $R$-modules so that $\delta(ab)=\delta(a)b+a\delta(b)$
	for all $a,b\in A$. We denote by $Der(A)$  the set of derivations of $A$.  It is well known that  $Der(A)$ is a  Lie algebra over $R$
	under  the commutator of derivations  $[\delta_1,\delta_2]=\delta_1\circ \delta_2-\delta_2\circ \delta_1$.
	The $R$-module  $Der(A)$ is well known to be  a  $A$-module if  the algebra $A$ is commutative.
	\begin{definition}\cite[page 5]{R7}
		Let $R$ be a commutative ring, and let $A$ be a commutative $R$-algebra (not necessarily with unit). Let $(L, [-,-])$ be a Lie algebra over $R$. 
		A Lie-Rinehart algebra structure on $L$ is a Lie algebra homomorphism $\rho: L \longrightarrow Der(A)$ satisfying the following compatibility properties:
		
		\begin{enumerate}
			\item $\rho(a\alpha)(b) = a\rho(\alpha)(b)$,
			\item $[\alpha, a\beta] = a[\alpha, \beta] + \rho(\alpha)(a)\beta$,
		\end{enumerate}
		where $a, b \in A$ and $\alpha, \beta \in L$.
		
		A Lie-Rinehart algebra is a pair $(L, \rho)$ where $\rho$ is a Lie-Rinehart structure on $L$.
	\end{definition}
	The Koszul bracket is a fundamental concept in Poisson geometry, associated with a bivector field $\pi$ on a manifold $M$. It is defined as
	\begin{equation}
		[\alpha, \beta]_K = \mathcal{L}_{\alpha^{\#}}(\beta) - \mathcal{L}_{\beta^{\#}}(\alpha) - d\pi(\alpha, \beta) \quad \text{for every } \alpha, \beta \in \Omega^1(M).\label{t2}
	\end{equation}
	For more details, see \cite{R8}.

	\section{Result}
	%\paragraph{Main result}
	%We obtain the following  results:
	We begin this section with a more precise formalization of our space.
	\subsection{$\theta$-almost twisted Poisson manifold}
	\begin{definition}\label{def1}
		An $\theta$-almost twisted Poisson structure on $M$ is a bivector field $\pi$ together with a 3-form $\varphi$ and a closed 1-form $\theta$ on $M$ such that the following hold: $d\varphi = \theta \wedge \varphi$, $\pi^{\#}(\theta) = 0$, and
		\begin{align}\label{Eq2}
			\frac{1}{2}[\pi, \pi] &= \pi^{\#}(\varphi).
		\end{align}
		A manifold $M$ equipped with such a structure $(\pi, \varphi, \theta)$ is called a $\theta$-almost twisted Poisson manifold.
	\end{definition}
	%In the following,  $(M, \pi, \varphi, \theta)$ is $\theta$-almost twisted Poisson manifold.
	The following remark allows us to give various classes of $\theta$-almost twisted Poisson manifolds, among which are those induced by twisted Poisson manifolds $(M, \pi, \varphi)$ with both closed and non-closed $\varphi$.
	\begin{remark}~ 
		\begin{enumerate}
			\item Every Poisson manifold $(M, \pi)$ can be endowed with a twisted Poisson structure $(\pi, \varphi)$, and a twisted Poisson manifold $(M, \pi, \varphi)$ is a $0$-almost twisted Poisson manifold.
			\item
			Let $(M, \pi_0, \varphi_0)$ be a twisted Poisson manifold, and let $f \in C^\infty(M, \mathbb{R})$ be a nowhere vanishing function such that $\pi_0^{\#}(df) = 0$. For $\pi = f\pi_0$, $\theta = -f^{-1} df$, and $\varphi = f^{-1} \varphi_0$, the quadruplet $(M, \pi, \varphi, \theta)$ is a $\theta$-almost twisted Poisson manifold. A similar example can be constructed on $M \times \mathbb{R}$ with: $\pi = e^t \pi_0$, $\varphi = e^{-t} \varphi_0$, and $\theta = -dt$ where $t$ is the canonical coordinate on $\mathbb{R}$. Moreover, the map $\varphi$ is non-closed on $M \times \mathbb{R}$.
			\item \label{exem1.1.5} Let $(x_1, x_2, x_3, x_4, x_5)$ be a coordinate system in $\mathbb{R}^5$. Let $\theta = dx_5$, $\pi = f \partial_{x_1} \wedge \partial_{x_2} + g \partial_{x_3} \wedge \partial_{x_4}$, and
			\begin{align*}
				\varphi &= \partial_{x_1} g^{-1} dx_1 \wedge dx_3 \wedge dx_4 + \partial_{x_2} g^{-1} dx_2 \wedge dx_3 \wedge dx_4 \\
				&\quad + (\partial_{x_5} g^{-1} - g^{-1}) dx_3 \wedge dx_4 \wedge dx_5 + \partial_{x_3} f^{-1} dx_1 \wedge dx_2 \wedge dx_3 \\
				&\quad + \partial_{x_4} f^{-1} dx_1 \wedge dx_2 \wedge dx_4 + (\partial_{x_5} f^{-1} - f^{-1}) dx_1 \wedge dx_2 \wedge dx_5,
			\end{align*}
			where $f$ and $g$ are nowhere vanishing functions on $\mathbb{R}^5$. Then the triple $(\pi, \varphi, \theta)$ is a $\theta$-almost twisted Poisson structure on $\mathbb{R}^5$, and $\varphi$ is non-closed.
		\end{enumerate}
	\end{remark}

		\begin{fact} Let $m \leq 4$. If $(\pi, \varphi, \theta)$ is a $\theta$-almost twisted Poisson structure on a $m$-dimensional manifold $M$, then $\varphi$ is closed; that is, $(\pi, \varphi)$ is a twisted Poisson structure on $M$. Conversely, a twisted Poisson structure $(\pi, \varphi)$ on $M$ defines a $\theta$-parameter family $\left\{(\pi, \varphi, \theta),\; \pi^{\#}(\theta)=0\right\}$ of $\theta$-almost twisted Poisson structures. In other words, for a fixed $m$-dimensional manifold $M$, there is a one-to-one correspondence between the set of twisted Poisson structures and the set of $\theta$-parameter families of $\theta$-almost twisted Poisson structures.
		\end{fact}
		%\begin{proof}
		%	Note that in dimension less than 4, every 3-form is closed. Thus  we are reduced to prove the proposition in the case of dimension 4.\\
		%	Let $(M,\pi, \varphi, \theta)$ be a  4-dimensional $\theta$-almost twisted Poisson manifold which is not a Poisson manifold.  We are going to show
		%	that $\varphi$ is closed.\\ Locally, $\pi$ and $\varphi$ are expressed respectively  by
		%	$\pi = \sum_{1\leq i<j\leq 4}^{}f_{ij}\partial_{x_i}\wedge\partial_{x_j}$, and $\varphi = \sum_{1\leq i<j\leq 4}^{}\varphi_{ij}d{x_i}\wedge
		%	d{x_j}$ where $f_{ij},\varphi_{ij}$ are smooth functions on $M$.  According to the equality  $[\pi,\pi]=\pi^{\#}(\varphi)$, and of the fact that
		%	$(M,\pi, \varphi, \theta)$ is not a Poisson manifold, we get
		%	
		%	\[\left\{\begin{array}{ccc}
			%	\varphi_{123}  &=& -\partial_{x_1}(f_{14}f^{-1})-\partial_{x_2}(f_{24}f^{-1})
			%	-\partial_{x_3}(f_{34}f^{-1})\\
			%	\varphi_{124} &=&
			%	\partial_{x_1}(f_{13}f^{-1})+\partial_{x_2}(f_{23}f^{-1})
			%	-\partial_{x_4}(f_{34}f^{-1})\\
			%	\varphi_{134}&=&
			%	-\partial_{x_1}(f_{12}f^{-1})+\partial_{x_3}(f_{23}f^{-1})
			%	+\partial_{x_4}(f_{24}f^{-1})\\
			%	\varphi_{234} &=&
			%	-\partial_{x_2}(f_{12}f^{-1})-\partial_{x_3}(f_{13}f^{-1})
			%	-\partial_{x_4}(f_{14}f^{-1}).
			%	\end{array}\right.\]
		%	
		%	It follows  that   $\partial_{x_1}\varphi_{234}-\partial_{x_2}\varphi_{134}+\partial_{x_3}\varphi_{124}
		%	-\partial_{x_4}\varphi_{123}=0$. This ends the proof.
		%\end{proof}

		\subsection{$\theta$-almost twisted Poisson  cohomology.}
		The bracket $[\ ,\ ]_{\varphi,\theta}: \Omega^1(M) \times \Omega^1(M) \longrightarrow \Omega^1(M)$ associated with a $\theta$-almost twisted Poisson manifold $(M, \pi, \varphi, \theta)$ is given by
		\begin{equation}
			[\alpha, \beta]_{\varphi, \theta} = [\alpha, \beta]_K + i_{\pi^{\#}(\beta)} i_{\pi^{\#}(\alpha)} \varphi + \pi(\alpha, \beta) \cdot \theta, \label{expresscrochet}
		\end{equation}
		where $\alpha, \beta \in \Omega^1(M)$. 
		
		Observe that  the associated bracket $[\ ,\ ]_{\varphi,\theta}$ of a $\theta$-almost twisted Poisson manifold is a $\theta$-almost twisted version of the associated bracket $[\ ,\ ]_{\varphi}:\ \Omega^1(M)\times\Omega^1(M)  \longrightarrow \Omega^1(M)$,
		$(\alpha,\beta) \longmapsto [\alpha,\beta]_K+i_{\pi^{\#}(\beta)}i_{\pi^{\#}(\alpha)}\varphi$ of a  twisted Poisson manifold which is also a twisted version of that of Koszul bracket $[\ ,\ ]_K$, see \eqref{t2}.  
		
		For a twisted Poisson manifold $(M,\pi,\varphi)$, the associated bracket $[\ ,\ ]_{\varphi}$ defines a Lie algebra structure on $\Omega^1(M)$. Moreover, the ordered pair $(\Omega^1(M),\pi^{\#})$ is a Lie-Rinehart algebra, which induces a cochain complex whose associated cohomology is called Lichnerowicz-twisted Poisson cohomology, see \cite{R11,R2,R20}.
		The rest of this work is devoted to define $\theta$-almost twisted Poisson cohomology.

		\begin{remark}
			Let $(M,\pi,\varphi,\theta)$  be a $\theta$-almost twisted Poisson manifold.
			The map
			\begin{eqnarray}
				% \nonumber % Remove numbering (before each equation)
				\{\ ,\ \}: C^\infty(M)\times (C^\infty(M)&\longrightarrow& (C^\infty(M),\quad   \{f,g\}\mapsto \pi(df,dg)
			\end{eqnarray} is $\mathbb{R}-$bilinear and skew-symmetric. Moreover, for every $f,g,h \in C^\infty(M)$, we have
			\begin{align}\label{Jac}
				% \nonumber % Remove numbering (before each equation)
				\{f,\{g,h\}\}+\{g,\{h,f\}\}+\{h,\{f,g\}\}&=\pi^{\#}(\varphi)(df,dg,dh).
			\end{align}
		\end{remark}

		\begin{proposition}\label{inc4}
			Let  $\alpha, \beta\in\Omega^1(M)$, and $f,g \in C^\infty(M)$. We have
			\begin{align}\label{inc1}
				[f\alpha,g\beta]_{\varphi,\theta}&=fg[\alpha,\beta]_{\varphi,\theta}+\pi^{\#}(f\alpha)(g)\beta-\pi^{\#}(g\beta)(f)\alpha.
			\end{align}
		\end{proposition}
		\begin{proof}
			From \eqref{expresscrochet}, we have:  \begin{align*}
				[f\alpha,g\beta]_{\varphi,\theta}&=[f\alpha,g\beta]_K +
				i_{\pi^{\#}(g\beta)}i_{\pi^{\#}(f\alpha)}\varphi +
				\pi(f\alpha,g\beta)\cdot
				\theta\\
				&=\mathcal{L}_{\pi^{\#}(f\alpha)}(g\beta)-\mathcal{L}_{\pi^{\#}(g\beta)}(f\alpha)-d\pi(f\alpha,g\beta)\\
				&\quad+fgi_{\pi^{\#}(\beta)}i_{\pi^{\#}(\alpha)}\varphi +
				fg\pi(\alpha,\beta)\cdot
				\theta\\
				&=i_{\pi^{\#}(f\alpha)}d(g\beta)-i_{\pi^{\#}(g\beta)}d(f\alpha)+d\pi(f\alpha,g\beta)\\
				&\quad+fgi_{\pi^{\#}(\beta)}i_{\pi^{\#}(\alpha)}\varphi +
				fg\pi(\alpha,\beta)\cdot \theta\\
				&=fg\big[i_{\pi^{\#}(\alpha)}d(\beta)-i_{\pi^{\#}(\beta)}d(\alpha)+d\pi(\alpha,\beta)\big]\\
				&\quad+\pi^{\#}(f\alpha)(g)\beta-\pi^{\#}(g\beta)(f)\alpha+
				fgi_{\pi^{\#}(\beta)}i_{\pi^{\#}(\alpha)}\varphi +
				fg\pi(\alpha,\beta)\cdot \theta\\
				&=fg[\alpha,\beta]_{\varphi,\theta}+\pi^{\#}(f\alpha)(g)\beta-\pi^{\#}(g\beta)(f)\alpha.
			\end{align*}
			This completes the proof.
		\end{proof}
		The bracket $[\ ,\ ]_{\varphi,\theta}$ is related to the anchor map $\pi^{\#}$. The precise formalization is as follows: 
		\begin{proposition}\label{inc2}For every $f,g,u ,v,c\in C^\infty(M)$, we  have
			\begin{align*}
				\pi^{\#}\big([udf,vdg]_{\varphi,\theta}\big)&=[\pi^{\#}(udf),\pi^{\#}(vdg)],
			\end{align*}
		\end{proposition}
		\begin{proof}
			Let $f,g,u ,v,c\in C^\infty(M)$. We get
			\begin{align}\label{inc3}
				\nonumber
				\pi^{\#}\big([udf,vdg]_{\varphi,\theta}\big)(c)&=\pi^{\#}\Big(uv[df,dg]_ {\varphi,\theta}+\pi^{\#}(udf)(v)dg-\pi^{\#}(vdg)(u)df\Big)(c)\\ \nonumber
				&=uv\pi^{\#}\Big([df,dg]_
				{\varphi,\theta}\Big)(c)+\pi^{\#}(udf)(v)\pi^{\#}(dg)(c)\\ \nonumber
				&\quad-\pi^{\#}(vdg)(u)\pi^{\#}(df)(c)\\ \nonumber
				&=uv\pi^{\#}\big(d\{df,dg\}+i_{\pi^{\#}(dg)}i_{\pi^{\#}(df)}\varphi\big)(c)\\ \nonumber
				&\quad+\pi^{\#}(udf)(v)\pi^{\#}(dg)(c)-\pi^{\#}(vdg)(u)\pi^{\#}(df)(c)\\ \nonumber
				&=uv[\pi^{\#}(df),\pi^{\#}(dg)](c)+\pi^{\#}(udf)(v)\pi^{\#}(dg)(c)\\
				&\quad-\pi^{\#}(vdg)(u)\pi^{\#}(df)(c)\\  \nonumber
				&=[\pi^{\#}(udf),\pi^{\#}(vdg)](c)\end{align}
			where equality  \eqref{inc3} comes from  Proposition \ref{prop1}.
			The proposition is then proved.
		\end{proof}
		
		In  the following,   we set  $\eta_{f,g,h}=i_{(df)^{\#}}i_{(dg)^{\#}}i_{(dh)^{\#}}\varphi$, and  $\eta_{f,g}=i_{(df)^{\#}}i_{(dg)^{\#}}\varphi$
		for every  $f, g, h \in C^\infty(M)$. The two following results are crucial for defining a Lie-Rinehart structure on $\Omega^1(M)$, as they provide the necessary conditions for the structure.
		\begin{lemma}\label{lemm1} For every real functions $f$, $g$, and $h$ on $M$,
			$$[df,\eta_{h,g}]_{\varphi,\theta}=i_{(df)^{\#}}d\eta_{h,g}+d\eta_{f,h,g}+ i_{[(dg)^{\#},(dh)^{\#}]}i_{(df)^{\#}}\varphi-\eta_{\{g,h\},f}+\eta_{f,h,g}\theta.$$
		\end{lemma}
		\begin{proof}	
			From \eqref{expresscrochet}, it follows that
			\begin{align*}
				% \nonumber % Remove numbering (before each equation)
				[df,\eta_{h,g}]_{\varphi,\theta} &=[df,\eta_{h,g}]_K+i_{(\eta_{h,g})^{\#}}i_{(df)^{\#}}\varphi +\pi(df,\eta_{h,g})\theta\\
				&=\mathcal{L}_{(df)^{\#}}(\eta_{h,g})-\mathcal{L}_{(\eta_{h,g})^{\#}}(df)-d\pi(df,\eta_{h,g})
				+i_{(\eta_{h,g})^{\#}}i_{(df)^{\#}}\varphi +\eta_{f,h,g}\theta\\
				&=i_{(df)^{\#}}d\eta_{h,g}+di_{(df)^{\#}}\eta_{h,g}-di_{(\eta_{h,g})^{\#}}df-d\pi(df,\eta_{h,g})\\
				&\quad+i_{(\eta_{h,g})^{\#}}i_{(df)^{\#}}\varphi+\eta_{f,h,g}\theta\\
				&=i_{(df)^{\#}}d\eta_{h,g}+di_{(df)^{\#}}\eta_{h,g}+d\pi(df,\eta_{h,g})-d\pi(df,\eta_{h,g})\\
				&\quad+i_{(\eta_{h,g})^{\#}}i_{(df)^{\#}}\varphi+\eta_{f,h,g}\theta\\
				&=i_{(df)^{\#}}d\eta_{h,g}+di_{(df)^{\#}}\eta_{h,g}+i_{(\eta_{h,g})^{\#}}i_{(df)^{\#}}\varphi+\eta_{f,h,g}\theta.
			\end{align*}
			According to  Proposition~\ref{prop1}, we have  $$(\eta_{h,g})^{\#}=[(dg)^{\#},(dh)^{\#}]-\big(d\{g,h\}\big)^{\#}.$$
			Therefore,
			\begin{equation}\label{sr3}
				\begin{array}{l}
					[df,\eta_{h,g}]_{\varphi,\theta}= i_{(df)^{\#}}d\eta_{h,g}+di_{(df)^{\#}}\eta_{h,g}+i_{[(dg)^{\#},(dh)^{\#}]}i_{(df)^{\#}}\varphi
					\\\hspace{2cm}-i_{(d\{h,g\})^{\#}}i_{(df)^{\#}}\varphi+\eta_{f,h,g}\theta
					i_{(df)^{\#}}d\eta_{h,g}+d\eta_{f,h,g}\\\hspace{2cm}+ i_{[(dg)^{\#},(dh)^{\#}]}i_{(df)^{\#}}\varphi-\eta_{\{g,h\},f}+\eta_{f,h,g}\theta.
				\end{array}
			\end{equation}	
			This is precisely the assertion of Lemma~\ref{lemm1}.
			
		\end{proof}\par
		\begin{lemma}\label{lem2} For all  functions $f$, $g$ and $h$   on $M$, we have
			\begin{align*}
				[df,[dg,dh]_{\varphi,\theta}]_{\varphi,\theta} & =d\{f,\{g,h\}\}+\{f,\{g,h\}\}\theta+\eta_{\{g,h\},f}\\
				&\quad+[df,\eta_{h,g}]_{\varphi,\theta}+[df,\{g,h\}\theta]_{\varphi,\theta}.
			\end{align*}	
		\end{lemma}
		\begin{proof}
			From \eqref{expresscrochet}, we have
			\begin{align*}
				% \nonumber % Remove numbering (before each equation)
				[dg,dh]_{\varphi,\theta} &= [dg,dh]_K+i_{(dh)^{\#}} i_{(dg)^{\#}}\varphi+ \pi(dg,dh) \theta\\
				&= d\{g,h\}+\eta_{h,g}+\{g,h\}\theta.
			\end{align*}
			Then,
			\begin{align*}
				%	\label{t3}
				\nonumber % Remove numbering (before each equation)
				[df,[dg,dh]_{\varphi,\theta}]_{\varphi,\theta}
				&= [df,d\{g,h\}]_{\varphi,\theta}+[df,\eta_{h,g}]_{\varphi,\theta} + [df,\{g,h\}\theta]_{\varphi,\theta} \\ \nonumber
				&=[df,d\{g,h\}]_K+i_{(d\{g,h\})^{\#}}i_{(df)^{\#}}\varphi +[df,\eta_{h,g}]_{\varphi,\theta}\\ \nonumber
				&\quad +\pi(df,d\{g,h\})\theta+[df,\{g,h\}\theta]_{\varphi,\theta}\nonumber\\
				&=d\{f,\{g,h\}\}+\eta_{\{g,h\},f}+\{f,\{g,h\}\}\theta \nonumber\\ 
				&\quad+[df,\eta_{h,g}]_{\varphi,\theta}+[df,\{g,h\}\theta]_{\varphi,\theta}.\nonumber
			\end{align*}
		\end{proof}
		Now we have the tools to prove our main result, which states more precisely as follows:
		\begin{theorem}\label{prop3}
			Let $(M,\pi,\varphi,\theta)$ be a $\theta$-almost twisted Poisson manifold. The bracket $[\ ,\ ]_{\varphi,\theta}$ defined in \eqref{expresscrochet} is  a  Lie bracket on $\Omega^1(M)$. Moreover, the morphism  $\pi^{\#} :\Omega^1(M)\longrightarrow \mathfrak{X}^1(M)$ defined  in $\eqref{expPi}$ is  a Lie-Rinehart  structure on $\Omega^1(M)$.
		\end{theorem}
		\begin{proof}~ 
			\begin{itemize}
				\item Let us first prove that the bracket $[\ ,\ ]_{\varphi,\theta}$ is a Lie bracket on $\Omega^1(M)$.
				The $\mathbb{R}$-bilinearity and skew-symmetry properties of the bracket $[\ ,\ ]_{\varphi,\theta}$ are immediate from its definition.
				We are going to prove the Jacobi identity. Using equality \eqref{inc1} and Proposition \ref{prop1}, we get
				$$
				[udf,[vdg,wdh]_{\varphi,\theta}]_{\varphi,\theta}+\circlearrowleft=uvw\big([df,[dg,dh]_{\varphi,\theta}]_{\varphi,\theta}+\circlearrowleft\big) 
				$$
				for every $f,g,h,u,v,w\in C^\infty(M)$, where $\circlearrowleft$ means that we perform a cyclic sum over $f$, $g$, and $h$. Therefore, we are reduced to proving the Jacobi identity of $[\ ,\ ]_{\varphi,\theta}$ on Pfaffian forms.
				
				From \eqref{Jac}, we get
				\begin{align}\label{S1}
					% \nonumber % Remove numbering (before each equation)
					0 &=dJ_{\{\ ,\ \}}(f,g,h)-d\pi^{\#}(\varphi)(df,dg,dh),
				\end{align}
				where $J_{\{\ ,\ \}}(f,g,h)=\{f,\{g,h\}\}+\{g,\{h,f\}\}+\{h,\{f,g\}\}$.
				According to Lemma \ref{lem2} and relation \eqref{S1}, we have
				\begin{align}\label{sr2}
					% \nonumber % Remove numbering (before each equation)
					\nonumber
					0 &= [df,[dg,dh]_{\varphi,\theta}]_{\varphi,\theta}+[dg,[dh,df]_{\varphi,\theta}]_{\varphi,\theta}+[dh,[df,dg]_{\varphi,\theta}]_{\varphi,\theta}\\ \nonumber
					&\quad-[df,\eta_{h,g}]_{\varphi,\theta}-[dg,\eta_{f,h}]_{\varphi,\theta}-[dh,\eta_{g,f}]_{\varphi,\theta}\\ \nonumber
					&\quad-\eta_{\{g,h\},f}-\eta_{\{h,f\},g}-\eta_{\{f,g\},h}-d\pi^{\#}(\varphi)(df,dg,dh)\\ \nonumber
					&\quad-\{f,\{g,h\}\}\theta-\{g,\{h,f\}\}\theta-\{h,\{f,g\}\}\theta\\
					&\quad-[df,\{g,h\}\theta]_{\varphi,\theta}-[dg,\{h,f\}\theta]_{\varphi,\theta}-[dh,\{f,g\}\theta]_{\varphi,\theta}.
				\end{align}
				From \eqref{expresscrochet}, we have
				\begin{align}\label{sr1}
					\nonumber % Remove numbering (before each equation)
					[df,\{g,h\}\theta]_{\varphi,\theta} &=[df,\{g,h\}\theta]_K+i_{(\{g,h\}\theta)^{\#}}i_{(df)^{\#}}\varphi +\pi(df,\{g,h\}\theta)\theta\\ \nonumber
					&=\mathcal{L}_{(df)^{\#}}(\{g,h\}\theta)-\mathcal{L}_{(\{g,h\}\theta)^{\#}}(df)-d\pi(df,\{g,h\}\theta)\\ \nonumber
					&\quad+i_{\{g,h\}\theta)^{\#}}i_{(df)^{\#}}\varphi +\pi(df,\{g,h\}\theta)\theta\\ \nonumber
					&=i_{(df)^{\#}}d(\{g,h\}\theta)+ di_{(df)^{\#}}(\{g,h\}\theta)+\{g,h\}i_{\theta^{\#}}i_{(df)^{\#}}\varphi \\ \nonumber
					&\quad-\{g,h\}\theta^{\#}(df)\theta\\
					&=\{f,\{g,h\}\}\theta.
				\end{align}
				In the same way, we also have
				\begin{align}\label{sr10}
					% \nonumber % Remove numbering (before each equation)
					[dg,\{h,f\}\theta]_{\varphi,\theta} &= \{g,\{h,f\}\}\theta,
				\end{align}
				and
				\begin{align} \label{sr11}
					% \nonumber % Remove numbering (before each equation)
					[dh,\{f,g\}\theta]_{\varphi,\theta} &= \{h,\{f,g\}\}\theta.
				\end{align}
				
				According to Lemma~\ref{lemm1},  relations \eqref{sr1}, \eqref{sr10}, and  \eqref{sr11},  equality \eqref{sr2} becomes
				\begin{equation}\label{sr6}
					\begin{array}{l}
						0=[df,[dg,dh]_{\varphi,\theta}]_{\varphi,\theta}+[dg,[dh,df]_{\varphi,\theta}]_{\varphi,\theta}+[dh,[df,dg]_{\varphi,\theta}]_{\varphi,\theta}\\
						\hspace{0.7cm}-i_{(df)^{\#}}d\eta_{h,g}-d\eta_{f,h,g}-i_{[(dg)^{\#},(dh)^{\#}]}i_{(df)^{\#}}\varphi\\
						\hspace{0.7cm}- i_{(dg)^{\#}}d\eta_{f,h}-d\eta_{g,f,h}-i_{[(dh)^{\#},(df)^{\#}]}i_{(dg)^{\#}}\varphi\\
						\hspace{0.7cm}-i_{(dh)^{\#}}d\eta_{g,f}-d\eta_{h,g,f}-i_{[(df)^{\#},(dg)^{\#}]}i_{(dh)^{\#}}\varphi\\
						\hspace{0.7cm}-d\pi^{\#}(\varphi)(df,dg,dh)-\eta_{f,h,g}\theta.
					\end{array}
				\end{equation}
				Thus, for every  vector field $X$ on $M$, we have	
				\begin{equation}
					\begin{array}{l}
						0=i_{X}\Big([df,[dg,dh]_{\varphi,\theta}]_{\varphi,\theta}+[dg,[dh,df]_{\varphi,\theta}]_{\varphi,\theta}\\
						\hspace{0.7cm}+[dh,[df,dg]_{\varphi,\theta}]_{\varphi,\theta} \Big)-i_{X}i_{(df)^{\#}}d\eta_{h,g} \\ 
						\hspace{0.7cm}-i_{X}d\eta_{f,h,g}-i_{X}\Big(i_{[(dg)^{\#},(dh)^{\#}]}i_{(df)^{\#}}\varphi\Big)\\
						\hspace{0.7cm}- i_{X}i_{(db)^{\#}}d\eta_{f,h}-i_{X}d\eta_{g,fh}-i_{X}\Big(i_{[(dh)^{\#},(df)^{\#}]}i_{(dg)^{\#}}\varphi\Big)\\
						\hspace{0.7cm}-i_{X}i_{(dh)^{\#}}d\eta_{g,f}-i_{X}d\eta_{h,g,f}-i_{X}\Big(i_{[(df)^{\#},(dg)^{\#}]}i_{(dh)^{\#}}\varphi\Big)\\
						\hspace{0.7cm}-i_{X}\big(d\pi^{\#}(\varphi)(df,dg,dh)\big)-i_{X}\big(\eta_{f,h,g}\cdot\theta\big).
					\end{array}\label{sr7}	
				\end{equation}
				Furthermore,
				\begin{align}
					%		\begin{array}{l}
						i_{X}i_{(df)^{\#}}d\eta_{h,g}\nonumber &=(df)^{\#}\Big(\eta_{h,g}(X)\Big)-X\Big(\eta_{h,g}\big((df)^{\#}\big)\Big)\\\nonumber
						&\quad-
						\eta_{h,g}\Big(\big[(df)^{\#},X\big]\Big)\\ 
						&\begin{array}{l}=(df)^{\#}\Big(\varphi\big((dg)^{\#},(dh)^{\#},X\big)\Big)\\
							\hspace{0.4cm}-
							X\Big(\varphi\big((df)^{\#},(dg)^{\#},(dh)^{\#}\big)\Big)\\\hspace{0.4cm}-\varphi\Big(\big[(df)^{\#},X\big],(dg)^{\#},(dh)^{\#}\Big).
						\end{array}\label{ss1}	
					\end{align}
					In the same way, we have
					\begin{align}\label{ss2}
						% Remove numbering (before each equation)
						\begin{array}{l}i_{X}i_{(dg)^{\#}}d\eta_{f,h}= (dg)^{\#}\Big(\varphi\big((dh)^{\#},(df)^{\#},X\big)\Big)\\
							\hspace{2.5cm}-
							X\Big(\varphi\big((dg)^{\#},(dh)^{\#},(df)^{\#}\big)\Big)\\
							\hspace{2.5cm}
							-\varphi\Big(\big[(dg)^{\#},X\big],(dh)^{\#},(df)^{\#}\Big),
						\end{array}
					\end{align}
					and
					\begin{align}\label{ss3}
						% Remove numbering (before each equation)
						\begin{array}{l}i_{X}i_{(dh)^{\#}}d\eta_{g,f}=(dh)^{\#}\Big(\varphi\big((df)^{\#},(dg)^{\#},X\big)\Big)\\
							\hspace{2.5cm}-
							X\Big(\varphi\big((dh)^{\#},(df)^{\#},(dg)^{\#}\big)\Big)\\
							\hspace{2.5cm}-\varphi\Big(\big[(dh)^{\#},X\big],(df)^{\#},(dg)^{\#}\Big).
						\end{array}
					\end{align}
					From \eqref{sr7},   \eqref{ss1}, \eqref{ss2},  and  \eqref{ss3}, we obtain
					\begin{align*}
						0&=i_{X}\Big([df,[dg,dh]_{\varphi,\theta}]_{\varphi,\theta}+[dg,[dh,df]_{\varphi,\theta}]_{\varphi,\theta}+[dh,[df,dg]_{\varphi,\theta}]_{\varphi,\theta} \Big) \\
						&\quad-d\varphi\big((df)^{\#},(dg)^{\#},(dh)^{\#},X\big)+\theta\wedge\varphi\big((df)^{\#},(dg)^{\#},(dh)^{\#},X\big)\\
						&=i_{X}\Big([df,[dg,dh]_{\varphi,\theta}]_{\varphi,\theta}+[dg,[dh,df]_{\varphi,\theta}]_{\varphi,\theta}+[dh,[df,dg]_{\varphi,\theta}]_{\varphi,\theta} \Big).
					\end{align*}
					This prove the Jacobi identity of $[\ ,\ ]_{\varphi,\theta}$.
					\item  Proposition \ref{inc4} shows  that  $[\alpha,f\beta]_{\varphi,\theta}=f[\alpha,\beta]_{\varphi,\theta}+\pi^{\#}(\alpha)(f).\beta $,
					for every  1-forms $\alpha$, $\beta$, and function $f$ on $M$.
					\item  Proposition \ref{inc2}  establishes  that the  $C^\infty(M)$-module $\Omega^1(M)$  endowed with a bracket $[\ ,\ ] _{\varphi,\theta}$
					is a  $\mathbb{R}$-Lie algebra homomorphic to  $\mathfrak{X}^1(M)$ endowed with its natural Lie bracket.
				\end{itemize}
				This ends the proof.
			\end{proof}
			
			%The following remark allows us to define the  $\theta$-almost twisted Poisson  cohomology groups.
			
			%\begin{remark}
			%\end{remark}

			The key ingredient in defining the $\theta$-almost twisted Poisson cohomology of $(M,\pi,\varphi,\theta)$ is the differential operator $\partial_{\varphi,\theta}$ associated with $\mathfrak{X}(M)$, which takes the following form: 
			% As in the Poisson and twisted Poisson cases, $\partial_{\varphi,\theta}$ is canonically defined by $\partial_{\varphi,\theta}=[\ \pi,\ ] _{\varphi,\theta}$. 
			$$
			\partial_{\varphi,\theta}^n:\mathfrak{X}^{n}(M) \longrightarrow  \mathfrak{X}^{n+1}(M), \quad v \longmapsto \partial_{\varphi,\theta}^n(v)
			$$
			with
			\begin{equation}\label{expi1}
				\begin{array}{r}
					\partial_{\varphi,\theta}^n(v)(\alpha_1,\dots,\alpha_{n+1}) = \sum_{i=1}^{n+1}(-1)^{i-1}(\alpha_i)^{\#}(v(\alpha_1,\dots,\widehat{\alpha_i},\dots,\alpha_{n+1}))\\  \\+ \sum_{1\leq i<j\leq n+1}(-1)^{i+j}v([\alpha_i,\alpha_j]_{\varphi,\theta},\alpha_1,\dots,\widehat{\alpha_i},\dots,\widehat{\alpha_j},\dots,\alpha_{n+1})
				\end{array}
			\end{equation}
			for every $\alpha_1,\dots,\alpha_{n+1}\in\Omega^1(M)$. The Jacobi identity for the bracket $[\ ,\ ]_{\varphi,\theta}$ ensures that $\partial_{\varphi,\theta} \circ \partial_{\varphi,\theta} = 0$, a crucial property for the cohomology theory to be well-defined. As is standard, terms with a hat symbol are omitted.
			
			%
			%
			%The ordered pair $(\mathfrak{X}(M), \partial_{\varphi,\theta})$ is  a cochain complex. This is a similar result in  \cite[page 5-6]{R7}.
			
			\begin{definition}  The $\theta$-almost twisted Poisson cohomology of $(M,\pi,\varphi,\theta)$ is defined as the cohomology of the complex $(\mathfrak{X}(M), \partial_{\varphi,\theta})$, denoted by $H_{\theta-atP}^{*}(M,\pi,\varphi,\theta)$ or simply $H_{\theta-atP}^{*}(M)$ when there is no ambiguity. Specifically, for every positive integer $n$,
				\begin{equation}\label{t4}
					H_{\theta-atP}^{n}(M)=\dfrac{ker\Big(\partial_{\varphi,\theta}^{n}:\mathfrak{X}^{n}(M)\to \mathfrak{X}^{n+1}(M)\Big)}{Im\Big(\partial_{\varphi,\theta}^{n-1}:
						\mathfrak{X}^{n-1}(M)\to \mathfrak{X}^{n}(M)\Big)}
				\end{equation}
				where $\mathfrak{X}^{-1}(M)=\{0\}$ by convention. The cohomology class of any element $\mu \in \ker(\partial_{\varphi,\theta}^{n}:\mathfrak{X}^{n}(M)\to \mathfrak{X}^{n+1}(M))$ is denoted by $[\mu]_{\varphi,\theta}$.		
			\end{definition}
			
			Analogously to the approach in \cite[pages 41-43]{R19}, we have $\partial_{\varphi,\theta}\circ \pi^{\#}=-\pi^{\#}\circ d$, which implies the commutativity of the following diagram:
			\begin{center}
				\begin{tikzpicture}
					\node(a) at (0,0) {$0$};
					\node (b) at (2,0) {$C^\infty(M)$};
					\node (c) at (6,0) {$\Omega^1(M)$};
					\node (d) at (10,0) {$\Omega^2(M)$};
					\node (e) at (12,0) {$\dots$};
					
					\node(aa) at (0,-2) {$0$};
					\node(bb) at (2,-2) {$C^\infty(M)$};
					\node(cc) at (6,-2) {$\mathfrak{X}^1(M)$};
					\node(dd) at (10,-2) {$\mathfrak{X}^2(M)$};
					\node(ee) at (12,-2) {$\dots$};
					
					\draw[->] (a)--(b);
					\draw[->] (b)--(c) node[midway, sloped,above]{$-d$};
					\draw[->] (c)--(d) node[midway, sloped,above]{$-d^{2}$};
					\draw[->] (d)--(e);
					\draw[->] (aa)--(bb);
					\draw[->] (bb)--(cc)node[midway, sloped,above]{$\partial_{\varphi,\theta}^{1}$};
					\draw[->] (cc)--(dd)node[midway, sloped,above]{$\partial_{\varphi,\theta}^{2}$};
					\draw[->] (dd)--(ee);
					\draw[->] (b)--(bb) node[midway,left]{$\pi^{\#}$};
					\draw[->] (c)--(cc) node[midway,left]{$\pi^{\#}$};
					\draw[->] (d)--(dd) node[midway, left]{$\pi^{\#}$};
				\end{tikzpicture}
			\end{center}
			And we obtain the following result:
			\begin{proposition} Let $(M,\pi,\varphi,\theta)$ be a $\theta$-almost twisted Poisson manifold. The de Rham cohomology of $M$ is denoted by $H_{dR}^{*}(M,\mathbb{R})$. The cochain complex homomorphism $\pi^{\#}: (\Omega(M),d)\longrightarrow (\mathfrak{X}(M), \partial_{\varphi,\theta})$ induces a homomorphism in cohomology, also denoted by $\pi^{\#}$,
				\begin{equation}\label{ex1}
					\pi^{\#}: H_{dR}^{*}(M,\mathbb{R}) \longrightarrow  H_{\theta-atP}^{*}(M),\
					[\mu] \longmapsto [\pi^{\#}(\mu)]_{\varphi,\theta}.
				\end{equation}
				Moreover, if $\pi$ is nondegenerate, then $\pi^{\#}: H_{dR}^{*}(M,\mathbb{R}) \longrightarrow  H_{\theta-atP}^{*}(M)$ is an isomorphism.
			\end{proposition}
			
			\subsection{A  class of $\theta$-almost twisted Poisson cohomology  on $\mathbb{R}^5$}\label{TATPs1}
			Let  $(x_1,x_2,x_3,x_4,x_5)$ be a   coordinate system on  $\mathbb{R}^5$. We define:
			\begin{align*}
				\pi &= f \partial_{x_1} \wedge \partial_{x_2} + g \partial_{x_3} \wedge \partial_{x_4},\quad \theta = dx_5,\\
				\varphi &= \partial_{x_1} g^{-1} dx_1 \wedge dx_3 \wedge dx_4 + \partial_{x_2} g^{-1} dx_2 \wedge dx_3 \wedge dx_4
				\\ &\quad+ (\partial_{x_5} g^{-1} - g^{-1}) dx_3 \wedge dx_4 \wedge dx_5 
				\quad + \partial_{x_3} f^{-1} dx_1 \wedge dx_2 \wedge dx_3 \\ &\quad+ \partial_{x_4} f^{-1} dx_1 \wedge dx_2 \wedge dx_4
				+ (\partial_{x_5} f^{-1} - f^{-1}) dx_1 \wedge dx_2 \wedge dx_5,
			\end{align*}
			where $f$ and $g$ are nowhere vanishing functions on $\mathbb{R}^5$. The quadruplet $(\mathbb{R}^5,\pi,\varphi,\theta)$ is a  $\theta$-almost twisted Poisson manifold. After a rather tedious calculation, we obtain
			\begin{equation*}
				H_{\theta\text{-atp}}^p(\mathbb{R}^5) =\begin{cases}
					\begin{array}{lcl}
						C_{x^5}&\text{if} &p=0\\
						E_{\partial_{x_5}}&\text{if}& p=1\\
						0& \text{if}& p \geq 2
					\end{array}
				\end{cases}
			\end{equation*}
			with
			$$C_{x^5}=\{f\in C^\infty(\mathbb{R}^5), \;f(x_1,x_2,x_3,x_4,x_5)=f(x_5)\},\;  E_{\partial_{x_5}}= \{f \partial_{x_5},\; f \in C_{x^5}\}.$$

				\end{document}